\documentclass[10pt]{amsart}
\renewcommand{\Re}{\operatorname{Re}}
\renewcommand{\Im}{\operatorname{Im}}
\newcommand{\C}{\mathbb{C}}
\newcommand{\D}{\mathbb{D}}
\newcommand{\Z}{\mathbb{Z}}
\newcommand{\F}{\mathcal{F}}
\newcommand{\T}{\mathbb{T}}
\newcommand{\conj}[1]{\overline{#1}}
\newcommand{\<}{\langle}
\renewcommand{\hat}[1]{\widehat{#1}}
\renewcommand{\>}{\rangle}
\newcommand{\adj}[1]{#1^{*}}
\begin{document}
\newtheorem{theorem}{Theorem}
\newtheorem{proposition}{Proposition}
\newtheorem{lemma}{Lemma}
\newtheorem*{definition}{Definition}
\newtheorem{corollary}{Corollary}
\newtheorem{remark}{Remark}
\title{Unions of arcs from Fourier partial sums}
\begin{abstract}
Elementary complex analysis and Hilbert space methods show that a union of at most $n$ arcs on the circle is uniquely determined by the $n$th Fourier partial sum of its characteristic function.  The endpoints of the arcs can be recovered from the coefficients appearing in the partial sum by solving two polynomial equations.
\end{abstract}
\date{\today}
\author{Dennis Courtney}
\address{Department of Mathematics, University of Iowa, Iowa City, IA 52242}
\email{dennis-courtney@uiowa.edu}

\thanks{The author was partially supported by the University of Iowa Department of Mathematics NSF VIGRE grant DMS-0602242.}
\maketitle
\thispagestyle{empty}

We let $\T = \{z \in \C: |z| = 1\}$ and $\D = \{z \in \C: |z| < 1\}$, and for any subset $E$ of $\T$ and integer $k$ we write
$$
\hat{E}(k) = \frac{1}{2\pi} \int_E e^{-ikt} \, dt
$$
for the $k$th Fourier coefficient of the characteristic function $\chi_E$ of $E$.  As bounded functions with the same sequence of Fourier coefficients agree almost everywhere, any subset $E$ of $\T$ is determined up to a set of measure zero by the sequence $\hat{E}(k)$.  If $E$ is known to have additional structure, the entire sequence may not be needed to recover $E$.  Our present subject is a simple yet nontrivial illustration of this principle.

An \emph{arc} is by definition a closed, connected, proper and nonempty subset of $\T$.  We declare $\T$ along with the empty set to be a ``union of $0$ arcs.''  
\begin{theorem}\label{main}
If $n$ is a nonnegative integer and $E_1$ and $E_2$ are unions of at most $n$ arcs satisfying
\begin{equation}\label{four}
\hat{E_1}(k) = \hat{E_2}(k), \qquad 0 \leq k \leq n,
\end{equation}
then $E_1 = E_2$.
\end{theorem}
Thus a set $E$ that is known to be a union of at most $n$ arcs can be recovered \emph{completely} from the $n$th Fourier partial sum of $\chi_E$, regardless of any quantitative sense in which this partial sum fails to approximate $\chi_E$.  This stands in slight contrast to the well-known defects of Fourier partial sum approximation of functions with jump discontinuities, such as the Gibbs phenomenon (see e.g. \cite[Chapter 17]{kornerbook}).  Significantly, the property of the Fourier basis expressed by Theorem~\ref{main} is not shared by other orthonormal systems of functions on $\T$ (see \S\ref{section.toep}).

Our proof of Theorem~\ref{main} exploits a connection between unions of arcs and certain rational functions--- the \emph{Blaschke products}, whose properties we recall in \S\ref{section.blaschke}.  Each Blaschke product has a nonnegative integer \emph{order}.  In \S\ref{section.const} we construct an injection $E \mapsto b_E$ from the set of finite unions of arcs to the set of Blaschke products with the property that if $E$ is a union of at most $n$ arcs, then $b_E$ has order at most $n$.  This map has the property that if $E_1$ and $E_2$ satisfy \eqref{four}, then $b_{E_1}$ and $b_{E_2}$ have the same $n$th order Taylor polynomial at $0$.  To prove Theorem~\ref{main} it then suffices to note, as we do in \S\ref{section.toep}, how a Blaschke product of order at most $n$ is determined by its $n$th order Taylor polynomial.  

With Theorem~\ref{main} in hand, one may ask how to recover $E$ from a partial list of Fourier coefficients in an explicit fashion.  This is the subject of \S\ref{algorithms}, where we present an algorithm for testing whether or not a given tuple of complex numbers takes the form $(\hat{E}(k))_{k=0}^n$ for a union $E$ of at most $n$ arcs, and for finding the endpoints of these arcs in terms of the Fourier coefficients in this case.  

Perhaps because of its elementary nature, we have not found Theorem~\ref{main} explicitly stated in the literature, although it is known, and the literature abounds with more general theorems on the reconstruction of a function from partial knowledge of its Fourier transform.  In \cite{newpaper} it is shown that a function on $\T$ that is piecewise constant on a partition of $\T$ into $m$ connected pieces may be recovered from its $m$th Fourier partial sum.  Note that Theorem~\ref{main} concludes slightly more from a much stronger hypothesis.

The argument we use is known to specialists.  The basic idea is to apply a conformal map into the disc and then the classical Caratheodory-Fejer theorem \cite{caratheodoryfejer}.  This is by no means the only approach to Theorem~\ref{main}.  It should be contrasted with what one may get by viewing \eqref{four} as a system of polynomial equations and solving it directly with algebra.

We are indebted to Donald Sarason for many valuable discussions, and to Mihalis Kolountzakis for drawing our attention to \cite{newpaper}.

\section{Blaschke products}\label{section.blaschke}
\begin{definition} A (finite) Blaschke product is a function of the form 
\begin{equation}\label{blaschkeformula}
b(z) = \lambda \prod_{j=1}^n \frac{z-a_j}{1 - \overline{a_j} z}
\end{equation}
for some nonnegative integer $n$, some $\lambda \in \T$, and some $a_1, \dots, a_n \in \D$.  The nonnegative integer $n$ is called the \emph{order} of the Blaschke product.  
\end{definition}
If $n=0$ we interpret the empty product as $1$.  The domain of a Blaschke product is either $\T$, $\D$, or the closure $\overline{\D}$ of $\D$, depending on context.  A Blaschke product is evidently a rational function that maps $\T$ to itself and has no poles in $\D$ (it suffices to check the case $n=1$).  It is well known that these properties characterize the Blaschke products.
\begin{proposition}\label{wuk} If a rational function $r$ maps $\T$ to itself and has no poles in $\D$, then it is a Blaschke product of order equal to the number $n$ of zeros of $r$ in $\D$, counted according to multiplicity.
\end{proposition}
\begin{proof}
We induct on $n$.  If $n=0$, then $r = q^{-1}$ for some polynomial $q$; write $q(z) = \sum_{k=0}^m q_k z^k$ with $q_m \neq 0$.  As $q(\T) \subseteq \T$ we have
$$
q(z)^{-1} = \overline{q(z)} = \overline{q((\overline{z})^{-1})} = \sum_{k=0}^m \overline{q_k} z^{-k} = \frac{\sum_{k=0}^m \overline{q_k} z^{m-k}}{z^m}, \qquad z \in \T,
$$
so this holds for all nonzero $z \in \D$.  As $q$ has no zeros in $\D$, the extreme right hand side has no pole at $0$; thus $m=0$ and $q$ is constant as desired.

If $r$ has $n+1$ zeros in $\D$, choose one, $a$, and note that $r(z) \cdot (\frac{z-a}{1-\overline{a}z})^{-1}$ has $n$ zeros in $\D$ and maps $\T$ to itself.
\end{proof}
\begin{definition}
If $b$ is a Blaschke product, we let $U_b = \{z \in \T: \Im z \geq 0\}$.
\end{definition}
If the zeros of a Blaschke product are $a_1, \dots, a_n$, we calculate from \eqref{blaschkeformula}
$$
\frac{z b'(z)}{b(z)} = \sum_{j=1}^n \frac{1 - |a_j|^2}{|z - a_j|^2} > 0, \qquad z \in \T,
$$
so the argument of $b(e^{it})$ is strictly increasing in $t$.  The argument principle implies that $b(e^{it})$ travels $n$ times counterclockwise around $\T$ as $t$ runs from $0$ to $2\pi$.  
\begin{corollary}\label{mapping}
A Blaschke product $b$ has order $n$ if and only if $U_b$ is a disjoint union of $n$ arcs.
\end{corollary}
This is the main reason we include $\T$ as a ``union of $0$ arcs.''

\section{Blaschke products from unions of arcs}\label{section.const}

Let $S = \{z \in \C: 0 \leq 2 \Re z \leq 1\}$ and let $\phi$ denote the function
$$
\phi(z) = \frac{\exp(2 \pi i (z-1/4)) - 1}{\exp(2 \pi i (z - 1/4)) + 1}.
$$
It is easy to show (see e.g. \cite[\S III.3]{conwaybook}) that $\phi$ maps $S$ bijectively onto $\overline{\D} \setminus \{\pm 1\}$, that $\phi$ restricts to an analytic bijection of the interior of $S$ with $\D$, that $\phi$ maps the right boundary line of $S$ onto $\{z \in \T: \Im z > 0\}$, and that $\phi$ maps the left boundary line of $S$ onto $\{z \in \T: \Im z < 0\}$.
\begin{proposition}\label{hprop}
If $E$ is a disjoint union of $n \geq 0$ arcs and $h_E$ is given by
\begin{equation}\label{definition.he}
h_E(z) = \frac{1}{2} \hat{E}(0) + \sum_{k=1}^{\infty} \hat{E}(k) z^k, \qquad z \in \D,
\end{equation}
then $h_E$ is an analytic map of $\D$ into $S$, and the function $\D \to \overline{\D}$ given by
$$
b_E = \phi \circ h_E
$$
extends uniquely to a Blaschke product $\overline{\D} \to \overline{\D}$ of order $n$ satisfying $U_{b_E} = E$.
\end{proposition}
Using the formulas for $\phi$ and $h_E$ one can show without much work that $b_E$ is a rational function; the work in proving Proposition~\ref{hprop} is to establish that $b_E$ has the mapping properties of Proposition~\ref{wuk}, and hence is a Blaschke product, and to prove that $U_{b_E} = E$.  

To motivate the argument, let us work nonrigorously for a moment.  Formally we have the series expansion
\begin{equation}\label{chie}
\chi_E(z) = \sum_{k \in \Z} \hat{E}(k) z^k, \qquad z \in \T,
\end{equation}
and formal manipulation of the series \eqref{definition.he} with $z \in \T$ then shows that 
$$
\chi_E(z) = h_E(z) + \overline{h_E(z)} = 2 \Re h_E(z), \qquad z \in \T.
$$
As $\chi_E$ is $\{0, 1\}$ valued on $\T$, the maximum principle for harmonic functions then implies that $h_E$ maps $\D$ into $S$, so $b_E = \phi \circ h_E$ maps $\overline{\D}$ into $\overline{\D}$ and sends the circle to itself.  By Proposition~\ref{wuk} it follows that $b_E$ is a Blaschke product; the equality $U_{b_E} = E$ comes from the mapping properties of $\phi$ on the boundary of $S$.  

What makes this argument nonrigorous is that the series \eqref{chie} does not converge for all $z \in \T$, and to equate $\chi_E$ with $2 \Re h_E$ is to ignore the distinction between a discontinuous real valued function on $\T$ and a harmonic function on $\D$.  To fill in these gaps, we need to use the actual connection between $2 \Re h_E$ and $\chi_E$--- the former is the Poisson integral of the latter.
\begin{proof}
It is easily checked that \eqref{definition.he} does define an analytic function on $\D$, e.g. because $\sum_{k=1}^{\infty}|\hat{E}(k)|^2$ is convergent.  One can then verify the identity
$$
2 h_E(z) = \frac{1}{2\pi} \int_0^{2\pi} \frac{1 + z e^{-is}}{1 - z e^{-is}} \chi_E(e^{is}) \, ds, \qquad z \in \D.
$$
(Fix $z$, expand $\frac{1}{1 - z e^{-is}}$ as a power series in $z$ and interchange the sum and the integral.)  Taking real parts it follows that for any $r \in [0,1)$ and any $t$
\begin{equation}\label{realpart}
2 \Re h_E(r e^{it}) = \frac{1}{2\pi} \int_0^{2\pi} P_r(t-s) \chi_E(e^{is}) \, ds, 
\end{equation}
where 
$$
P_r(t) = \Re \left(\frac{1 + r e^{it}}{1 - r e^{it}}\right)
$$
is the \emph{Poisson kernel}.  It is elementary (see e.g. \cite[\S X.2]{conwaybook}) that for $r \in [0,1)$ the function $P_r$ is nonnegative and satisfies $\frac{1}{2\pi} \int_0^{2\pi} P_r(\theta) \, d\theta = 1$; thus \eqref{realpart} implies that $2 \Re h_E(z) \in [0,1]$ for all $z \in \D$, and $h_E$ maps $\D$ into $S$.  

As $r$ increases to $1$, the $P_r$ converge uniformly to the zero function on the complement of any neighborhood of $0$ (see e.g. \cite[\S X.2]{conwaybook}).  From \eqref{realpart} we conclude
\begin{equation}\label{boundary}
\lim_{r \uparrow 1} 2 \Re h_E(rz) = \chi_E(z)
\end{equation}
at any $z \in \T$ at which $\chi_E$ is continuous.  We conclude that for any such $z$ the limit $\lim_{r \uparrow 1} (\phi \circ h_E)(rz)$ exists and is in $\T$.

We claim that $\phi \circ h_E$ is a rational function.  In the case $n=0$ this is clear.  Otherwise, from the definition of $\phi$ it suffices to show that $\exp(2 \pi i h_E)$ is a rational function, and for this it suffices to treat the case $n=1$.  In this case there are real numbers $a < b$ with $b-a<2\pi$ satisfying $E = \{e^{it}: t \in [a,b]\}$, and $\hat{E}(k) = \frac{\exp(-ikb) - \exp(-ika)}{-2 \pi ik}$ for all $k > 0$.  Let $\log$ denote the analytic logarithm defined on $\C \setminus \{z \in \C: z \leq 0\}$ that is real on the positive real axis and recall that $\log(1 - z) = -\sum_{k=1}^{\infty} \frac{z^k}{k}$ for all $z \in \D$.  A comparison of power series shows
$$
h_E(z) = \frac{b-a}{4\pi} + \frac{1}{2\pi i} \left(\log(1 - e^{-ib} z) - \log(1 - e^{-ia} z)\right), \qquad z \in \D,
$$
so $\exp(2 \pi i h_E) = \exp(i\frac{b-a}{2}) \frac{1 - e^{-ib} z}{1 - e^{-ia} z}$ is rational.

At this point we know that $b_E = \phi \circ h_E$ is a rational function mapping $\D$ into itself.  From \eqref{boundary} we deduce that $b_E$ maps $\T$ into itself, so $b_E$ is a Blaschke product by Proposition~\ref{wuk}.  The equality $U_{b_E} = E$ then follows from \eqref{boundary}.  The order of $b_E$ is $n$ by Corollary~\ref{mapping}.
\end{proof}
If $E_1$ and $E_2$ are two unions of arcs related by \eqref{four}, it is clear from the definition that $h_{E_1}$ and $h_{E_2}$ have the same $n$th order Taylor polynomial at $0$.  As $\phi$ is analytic at $0$, the same is true of $b_{E_1}$ and $b_{E_2}$.
\begin{corollary}\label{bigcorollary}
If $n \geq 0$ and $E_1$ and $E_2$ are each unions of at most $n$ arcs satisfying
\begin{equation}\label{hyp}
\hat{E_1}(k) = \hat{E_2}(k), \qquad 0 \leq k \leq n,
\end{equation}
then there are Blaschke products $b_1$ and $b_2$, each of order at most $n$, satisfying $E_j = U_{b_j}$ for $j = 1,2$ and
\begin{equation}\label{apply}
\hat{b_1}(k) = \hat{b_2}(k), \qquad 0 \leq k \leq n.
\end{equation}
\end{corollary}

\section{Blaschke products from Toeplitz matrices}\label{section.toep}
Fix a positive integer $n$ for the remainder of this section.  Our goal is to show that Blaschke products $b_1$ and $b_2$ having order at most $n$ and satisfying \eqref{apply} must be equal.  Let $L^2$ denote the space of square-integrable functions $\T \to \C$, with inner product
$$
\<f,g\> = \frac{1}{2\pi} \int_0^{2\pi} f(e^{it}) \overline{g(e^{it})} \, dt, \qquad f, g \in L^2.
$$
(We identify two functions if they agree almost everywhere.)  

For $0 \leq k \leq n$ we let $\zeta^k$ denote the function $\T \to \C$ given by $z \mapsto z^k$.  It is immediate that $\{\zeta^k: 0 \leq k \leq n\}$ is an orthonormal subset of $L^2$.  We denote its span, the space of analytic polynomials of degree at most $n$, by $P$; we let $\pi: L^2 \to P$ denote the orthogonal projection.  
\begin{definition}
If $f: \T \to \C$ is bounded, $T_f: P \to P$ denotes the linear map given by
$$
T_f \xi = \pi( f \xi), \qquad \xi \in P.
$$
Here $f \xi$ is the pointwise product of $f$ and $\xi$.  
\end{definition}
If we let $\|T_f\|$ denote the norm of $T_f$ regarded as a linear operator on $P$ and write $\|f\|_{\infty} = \sup_{z \in \T} |f(z)|$, it is clear that
$$
\|T_f\| \leq \|f\|_{\infty}
$$
for any bounded $f$.  It is also clear that for any such $f$ 
$$
\<T_f \zeta^k, \zeta^j\> = \hat{f}(j-k), \qquad 0 \leq j,k \leq n,
$$
so the matrix of $T_f$ with respect to the orthonormal basis $\{\zeta^k: 0 \leq k \leq n\}$ is constant along its diagonals (it is a \emph{Toeplitz matrix}).

If $f$ is a Blaschke product, then $f$ is analytic on $\overline{\D}$, so the matrix of $T_f$ is lower triangular with first column $(\hat{f}(k))_{k=0}^n$.  Our hypothesis \eqref{apply} is thus that $T_{b_1} = T_{b_2}$, and to deduce that $b_1 = b_2$ it suffices to show how to recover a Blaschke product $b$ of order at most $n$ from the operator $T_b$ it induces on $P$.
\begin{lemma}\label{biglemma}
If $b$ is a Blaschke product of order at most $n$, then $\|T_b\| = 1$, and for any nonzero $r \in P$ satisfying $\|T_b r\| = \|r\|$ one has $T_b r = b r$.
\end{lemma}
This proof is a special case of the proof of \cite[Proposition 5.1]{sarasoninterp}.
\begin{proof}
There are nonzero polynomials $p$ and $q$, each of degree at most $n$, satisfying $b = p/q$.  Clearly $T_b q = p$, and as $b$ maps $\T$ to itself, we have $|p(z)| = |q(z)|$ for all $z \in \T$, so $\|p\| = \|q\|$.  We deduce that $\|T_b q\| = \|q\|$ and thus $\|T_b\| \geq 1$; since also $\|T_b\| \leq \|b\|_{\infty} = 1$, we conclude $\|T_b\| = 1$.

If $r \in P$ satisfies $\|T_b r\| = \|r\|$ we have
$$
\|r\|^2 = \|T_b r\|^2 = \|\pi(br)\|^2 \leq \|br\|^2 = \int_0^{2\pi} |b(e^{it})|^2 |r(e^{it})|^2 \, dt = \|r\|^2,
$$
from which $\|\pi(br)\| = \|br\|$ and thus $\pi(br) = br$ as desired.
\end{proof}

\begin{remark}\label{carath}
The argument of Lemma~\ref{biglemma} can be modified to show that if $f$ is bounded and analytic on $\overline{\D}$ and $\|f\|_{\infty} = 1$, then $\|T_f\| \leq 1$ with equality if and only if $f$ is a Blaschke product of order at most $n$.  With more work, one can prove the rest of the classical Caratheodory-Fejer theorem: that every lower triangular $(n+1) \times(n+1)$ Toeplitz $M$ satisfying $\|M\| = 1$ is of the form $T_f$ for such an $f$.
\end{remark}
We can now prove Theorem~\ref{main}.
\begin{proof}[Proof of Theorem~\ref{main}]
By Corollary~\ref{bigcorollary} there are Blaschke products $b_1$ and $b_2$ of order at most $n$ satisfying $U_{b_j} = E_j$ for $j=1,2$ and $\hat{b_1}(k) = \hat{b_2}(k)$ for $0 \leq k \leq n$.  This second fact implies that $T_{b_1} = T_{b_2}$.  By Lemma~\ref{biglemma} there is nonzero $q \in P$ satisfying $\|T_{b_1} q\| = \|T_{b_2} q\| = \|q\|$ and 
$$
b_1 = \frac{T_{b_1} q}{q} = \frac{T_{b_2} q}{q} = b_2,
$$
so $E_1 = U_{b_1} = U_{b_2} = E_2$.
\end{proof}
As the Fourier coefficients of a bounded function are coefficients with respect to an orthonormal basis of the Hilbert space $L^2$, one might wonder if Theorem~\ref{main} is a special case of a simpler result about arbitrary orthonormal bases of $L^2$.  This is not the case.  There are, for example, orthonormal bases $B$ for $L^2$ with the property that for every finite subset $F \subseteq B$, there is an arc $A$ with the property that every element of $F$ is constant on $A$.  (The basis $(e^{2 \pi it} \mapsto f(t))_{f \in H}$, where $H$ is the \emph{Haar basis} of $L^2[0,1]$ constructed in \cite[\S III.1]{haar}, has this property.) In this situation, if $E \subseteq A$ and $E' \subseteq A$ are any two unions of arcs with the same total measure, one will have $\<\chi_E, f\> = \<\chi_{E'}, f\>$ for all $f \in F$: \emph{any} finite collection of coefficients with respect to $B$ must fail to distiguish infinitely many unions of $n$ arcs from one another.  

\section{An algorithm}\label{algorithms} 
Let $\F$ denote the map sending a union of at most $n$ arcs $E$ to the tuple $(\hat{E}(k))_{k=0}^n$ in $\C^{n+1}$.  Suppose $c = (c_k)_{k=0}^n$ is given, and we desire to know whether or not $c$ in the range of $\F$.  The arguments of the previous sections give us the following procedure.  (We use the orthonormal basis of \S\ref{section.toep} to identify linear operators on $P$ with $(n+1)\times(n+1)$ matrices.)
\begin{enumerate}
\item Calculate the $n$th Taylor polynomial at $0$ for $\phi(\frac{c_0}{2} + \sum_{k=1}^n c_k z^k)$, and make its coefficients the first column of a lower-triangular Toeplitz matrix $M$.  
\item Evaluate $\|M\|$.

If $\|M\| \neq 1$, then $c$ is not in the range of $\F$.
\item Otherwise $\|M\| = 1$ and by the Caratheodory-Fejer theorem (see Remark~\ref{carath}) there is a unique Blaschke product $f$ of order at most $n$ satisfying $M = T_f$.  Find $F = U_f$ (e.g. by solving $f(z) = \pm 1$ to get the endpoints of the arcs) and calculate the coefficients of the $n$th order Taylor polynomial at $0$ for $b_F$.

If these coefficients are the first column of $M$ then $b_F = f$ and $c = \F(F)$; otherwise $c$ is not in the range of $\F$.
\end{enumerate}
\begin{remark} The third step of the algorithm is necessary as the map $E \mapsto b_E$ from unions of $n$ arcs to Blaschke products of order $n$ is not surjective.  One can check, for example, that of the Blaschke products $b_t(z) = \frac{z^n - t}{1 - tz^n}$ for real $|t| < 1$, all of which satisfy $U_{b_t} = U_{b_0}$, only $b_0$ is in the range of $E \mapsto b_E$.
\end{remark}
If we know in advance that $c = \F(E)$ is in the range of $\F$, this algorithm can recover $E$ from $c$ in a somewhat explicit fashion.  The matrix $M$ constructed from $c$ is $T_{b_E}$; Lemma~\ref{biglemma} implies that if we choose a nonzero $q \in P$ satisfying $\|M q\| = \|q\|$, we will have $b_E = \frac{M q}{q}$.  If $q$ is chosen so as to have minimal degree, the polynomials $M q$ and $q$ will have no nontrivial common factors.  In this case the degree of $q$ is the order of $b_E$, and the endpoints of the arcs of $E$--- the solutions to $b_E(z) = 1$ and $b_E(z) = -1$--- are the roots of the polynomials $Mq - q$ and $Mq + q$.  A computer has no difficulty carrying out this procedure to find the arcs of $E$ to any given precision from the tuple $c = \F(E)$.

As this algorithm involves solving polynomial equations, we cannot expect symbolic formulas for these endpoints of the arcs of $E$ in terms of the Fourier coefficients $\hat{E}(k)$.  Formulas for the polynomials $Mq \pm q$, however, can be obtained with some effort.  The entries of $M$ are polynomials in $\exp(2 \pi i \hat{E}(0))$, $\hat{E}(1)$, \dots, $\hat{E}(n)$ with complex coefficients.  As $M$ has norm $1$, a vector $q$ will satisfy $\|Mq\| = \|q\|$ if and only if $q$ is an eigenvector for the self-adjoint matrix $\adj{M}M$ corresponding to the eigenvalue $1$; we can find such a $q$ by using Gaussian elimination, for example.  As the entries of $\adj{M} M$ are polynomials in the entries of $M$ and their complex conjugates, the coefficients of $q$ and $Mq \pm q$ will be rational functions in $\exp(2 \pi i \hat{E}(0))$, $\hat{E}(1)$, \dots, $\hat{E}(n)$ and their complex conjugates.  Cases may arise in computing $Mq \pm q$ symbolically: in row reducing the symbolic matrix $\adj{M} M - I$, one needs to know whether or not certain functions of the matrix entries are zero--- but explicit formulas can be obtained in every case.

We give one example.  Suppose that $E$ is a union of at most two arcs, with $\hat{E}(0)$, $\hat{E}(1)$, and $\hat{E}(2)$ given.  Write $E_0  = \exp(2\pi i\hat{E}(0))$ and $E_k = -2\pi i k \hat{E}(k)$ for $k=1,2$.  Carrying out the above procedure, one finds that if both $E_1$ and the denominator of 
$$
a = \frac{E_2 \conj{E_1} + 2 E_1 - E_1^2 \conj{E_1} - 2 E_1 E_0}{E_1^2 E_0 + E_2 E_0 - E_2 + E_1^2},
$$ 
are nonzero, then the starting points of the arcs of $E$ are the solutions $z$ of the equation
$$
z^2 - az + \left(\frac{\conj{E_1} + (1 - E_0) a}{E_1 E_0}\right) = 0.
$$
The endpoints of the arcs of $E$ are given by a similar formula.

\end{document}